\documentclass[a4paper,10pt]{amsart}
\usepackage{amsmath,amsthm,amssymb,amsfonts,enumerate,color,bbm}

\date{\today}

\newcommand{\supess}{\operatorname{ess\, sup}}

\newtheorem{thm}{Theorem}[section]
\newtheorem{prop}[thm]{Proposition}

\newtheorem{lem}[thm]{Lemma}

\theoremstyle{definition}

\allowdisplaybreaks

\numberwithin{equation}{section}

\author[\'O. Ciaurri]{\'Oscar Ciaurri}
\address{Departamento de Matem\'aticas y Computaci\'on\\
         Universidad de La Rioja\\
         26004 Logro\~no, Spain}
\email{oscar.ciaurri@unirioja.es}

\thanks{Research supported by grant PID2021-124332NB-C22 AEI, from Spanish Government}



\begin{document}

\title[Hankel transform transplantation]{Uniform weighted inequalities\\
for the  Hankel transform transplantation operator}

\begin{abstract}
In this paper we present uniform weighted inequalities for the Hankel transform transplantation operator. A weighted vector-valued inequality is also obtained. As a consequence we deduce an extension of a transference theorem due to Rubio de Francia.
\end{abstract}

\maketitle

\section{Introduction and main results}
Let $J_{\alpha}$ be the Bessel function of order $\alpha$. For appropriate functions defined on $(0,\infty)$, we define the Hankel transform of order $\alpha>-1$ as the integral operator given by
\[
\mathcal{H}_\alpha f(x)=\int_0^\infty f(y) J_{\alpha}(xy)(xy)^{1/2}\, dy, \qquad x>0.
\]
It is known that $\mathcal{H}_\alpha\circ \mathcal{H}_\alpha f=f$ and $\|\mathcal{H}_\alpha f\|_{L^2}=\|f\|_{L^2}$,  for any $f\in C_c^\infty(0,\infty)$.

The transplantation operator for the Hankel transform is defined as
\[
T_{\alpha}^{\beta}=\mathcal{H}_{\beta}\circ \mathcal{H}_\alpha, \qquad \alpha\not= \beta.
\]
The main target of this paper is the analysis of uniform weighted inequalities for this operator with weights in the Muckenhoupt class $A_p(0,\infty)$, see the next section for its precise definition.

In particular, we focus on the case $\alpha=a+k$ and $\beta=b+k$, with $a\not=b$, $a,b\ge -1/2$, and $k=0,1,2,\dots$. Considering the operator  $S_k^{a,b}=T_{a+k}^{b+k}$ and the spaces
\[
L^p(u)=\left\{f: \|f\|_{L^p(u)}<\infty\right\},
\]
where
\[
\|f\|_{L^p(u)}:=\left(\int_0^\infty |f(x)|^p u(x) \, dx\right)^{1/p},
\]
we prove the following result.
\begin{thm}
\label{thm:main1}
Let $a\not =b$, $a,b> -1/2$, $k=0,1,2\dots$, $1<p<\infty$, and $u\in A_p(0,\infty)$. Then, 
\[
\|S_k^{a,b} f\|_{L^p(u)}\le C\|f\|_{L^p(u)},
\] 
where the constant $C$ depend on $a$ and $b$ but not on $k$.
\end{thm}

The boundedness of the transplantation operator for the Hankel transform $T_\alpha^\beta$ in $L^p$ spaces with powers weights was analyzed in \cite{Guy}, \cite{Schin} and \cite{Stem}. More general weights, including weights in the $A_p(0,\infty)$  class for $\alpha,\beta\ge -1/2$, were introduced in \cite{NoSt-Toh}. A result for the transplantation operator on Hardy spaces was proved in \cite{Kanjin}.

The Hankel transform appears in a natural way in harmonic analysis. In fact, for $f(x)=g(|x|)$, with $x\in \mathbb{R}^n$, it is verified that its Fourier transform is also a radial function and it is given by
\[
\hat f(\xi)=|\xi|^{-(n-1)/2}\mathcal{H}_{(n-2)/2}(g(\cdot) y^{(n-1)/2})(|\xi|).
\]
The action of the Fourier transform of radial functions multiplied by spherical harmonics is related to the Hankel transform also, see the last section. From this relation we can prove a generalization of a transference result due to Rubio de Francia. Indeed, in \cite[Theorem 2.2]{Rubio} it was proved that the inequality
\[
\|T_mf\|_{L^2(\mathbb{R}^n, u(|\cdot|))}\le C \|f\|_{L^2(\mathbb{R}^n,u(|\cdot|))}
\]
implies 
\[
\|T_mf\|_{L^2(\mathbb{R}^{n+2j},u(|\cdot|))}\le C \|f\|_{L^2(\mathbb{R}^{n+2j},u(|\cdot|))}, \qquad j=1,2,\dots,
\]
where $u$ is a nonnegative measurable function on $(0,\infty)$ and $T_m$ is the multiplier for the Fourier transform given by
\[
\widehat{T_m f}(\xi)=m(|\xi|)\hat{f}(\xi),
\]
with $m$ being a bounded function on $(0,\infty)$. This result allow us to deduce the boundedness of $T_m$ from the low dimensional results with jumps of length two. Our generalization of the transference theorem avoids such jumps but we have to include the restriction $u\in A_2(0,\infty)$.

\begin{thm}
\label{thm:JL}
Let $n\ge 2$ and $u\in A_2(0,\infty)$. Then the inequality
\[
\|T_m f\|_{L^2(\mathbb{R}^n, u(|\cdot|))}\le C \|f\|_{L^2(\mathbb{R}^n, u(|\cdot|))} 
\]
implies
\[
\|T_m f\|_{L^2(\mathbb{R}^{n+d}, u(|\cdot|))}\le C \|f\|_{L^2(\mathbb{R}^{n+d}, u(|\cdot|))}, \qquad d=1,2,\dots. 
\]
\end{thm}  

\section{Proof of Theorem \ref{thm:main1}}
For $1<p<\infty$, a nonnegative and locally integrable function $u$ belongs to the $A_p(0,\infty)$ class when for any interval $(a,b)\subset (0,\infty)$ it is verified that
\[
\left(\frac{1}{b-a}\int_{a}^{b}u(x)\, dx\right)\left(\frac{1}{b-a}\int_{a}^{b}u(x)^{-q/p}\, dx\right)^{p/q}<\infty,
\]
where $q$ is the conjugate of $p$; i. e., $p^{-1}+q^{-1}=1$. The definition of $A_1(0,\infty)$ weights is given in terms of the maximal function but we will not use them.

It is well known that for a Calder\'on-Zygmund operator $T$ bounded from $L^2$ into itself, the inequality 
\[
\|Tf\|_{L^p(u)}\le C \|f\|_{L^p(u)}, \qquad f\in L^2\cap L^p(u),
\]
holds for $u\in A^p(0,\infty)$ and $1<p<\infty$. As an obvious consequence, the operator $T$ extends to a bounded operator to $L^p(u)$.

In \cite[Remark 5,2]{NoSt-Toh} was shown that $T_\alpha^\beta$ is bounded from $L^p(u)$ into itself for $\alpha,\beta\ge -1/2$ and $u\in A_p(0,\infty)$, so we can focus on the analysis of $S_k^{a,b}$ for $k\ge k_0$, where $k_0$ is a nonnegative integer (for example, we can suppose $k_0\ge 10$). Showing that $S_k^{a,b}$ is a Calder\'on-Zygmund operator with an uniform standard kernel the proof of Theorem \ref{thm:main1} will be completed.

Schindler \cite{Schin} showed that for $\alpha,\beta\ge -1/2$ the operator $T_\alpha^\beta$ can be written as
\[
T_\alpha^\beta f(x)=\text{P. V.}\int_{0}^{\infty} K_{\alpha}^\beta(x,y) f(y)\, dy+\cos\left((\beta-\alpha)\frac{\pi}{2}\right)f(x),
\]
where the kernel $K_\alpha^\beta$ is given by
\[
\frac{2\Gamma((\alpha+\beta+2)/2)}{\Gamma(\alpha+1)\,\Gamma((\beta-\alpha)/2)}x^{-(\alpha+3/2)}y^{\alpha+1/2}
{}_2F_1\left(\frac{\alpha+\beta+2}{2},\frac{\alpha-\beta+2}{2};\alpha+1;\left(\frac{y}{x}\right)^2\right)
\]
for $0<y<x$, and
\[
\frac{2\Gamma((\alpha+\beta+2)/2)}{\Gamma(\beta+1)\,\Gamma((\alpha-\beta)/2)}x^{\beta+1/2}y^{-(\beta+3/2)}
{}_2F_1\left(\frac{\alpha+\beta+2}{2},\frac{\beta-\alpha+2}{2};\beta+1;\left(\frac{x}{y}\right)^2\right)
\]
for $0<x<y$.

In \cite{NoSt-Toh} it is proved that the expression for the transplantation operator can be extended to the range $\alpha,\beta >-1$, but we will consider $\alpha,\beta\ge-1/2$ only. Moreover, in such paper \cite[Proposition 3.1]{NoSt-Toh} it is showed that for $\alpha\not=\beta$
\[
\langle \mathcal{H}_\beta\circ \mathcal{H}_\alpha f, g\rangle =\int_{0}^{\infty}K_{\alpha}^\beta(x,y)f(y)\overline{g(x)}\, dx\, dy,
\]
for $f,g\in C_c^\infty(0,\infty)$ with disjoint supports. Using that $T_\alpha^\beta$ is bounded from $L^2$ itself, with the estimates in the next proposition, we will deduce that $S_k^{a,b}$ is an uniform Calder\'on-Zygmund operator for some values of the parameters $a$ and $b$.

\begin{prop}
\label{prop:C-Z}
Let $a,b\ge -1/2$, such that $0<|a-b|\le 1$, and $k\ge k_0$, where $k_0$ is a nonnegative integer. Then, for $x,y>0$ and $x\not= y$,
\begin{equation}
\label{eq:T1-size}
|K_{a+k}^{b+k}(x,y)|\le \frac{C_1}{|x-y|}
\end{equation}
and
\begin{equation}
\label{eq:T1-smooth}
\left|\frac{\partial}{\partial x} K_{a+k}^{b+k}(x,y)\right|+\left|\frac{\partial}{\partial y} K_{a+k}^{b+k}(x,y)\right|\le\frac{C_2}{(x-y)^2},
\end{equation}
where the constants $C_1$ and $C_2$ depend on $a$ and $b$ but not on $k$.
\end{prop}   

The proof of this proposition is a consequence of the next lemma which is a variant of \cite[Lemma 5.3]{Ciau-Ron-Lag}.

\begin{lem}
\label{lem:estimate}
Let $c\ge -1/2$, $d,\lambda>0$, $\gamma>-1$, and $0<B<A$. Then
\[
\int_{0}^{1}\frac{s^\gamma (1-s)^{d+c-1/2}}{(A-Bs)^{d+c+\lambda+1/2}}\, ds\le \frac{C_{\gamma,\lambda}}{d^\lambda}\frac{1}{A^{c+1/2}B^d(A-B)^\lambda},
\]
where the constant $C_{\gamma,\lambda}$ depend on $\gamma$ and $\lambda$ but not on $c$ and $d$.
\end{lem}

This lemma will be proved in the last section.

\begin{proof}[Proof of Proposition \ref{prop:C-Z}]
The starting point of the proof is the integral representation \cite[16.6.1]{NIST}
\begin{equation}
\label{eq:int-2F1}
{}_2F_1(p,q;r;z)=\frac{\Gamma(r)}{\Gamma(q)\,\Gamma(r-q)}\int_{0}^{1}\frac{s^{q-1}(1-s)^{r-q-1}}{(1-zs)^p}\, ds,
\end{equation}
where $r>q>0$ and $0<z<1$.

By using \eqref{eq:int-2F1}, for $0<y<x$ we have
\[
K_{a+k}^{b+k}(x,y)=\frac{(a+b)/2+k}{\Gamma((b-a)/2)\, \Gamma((a-b+2)/2)}x^{b+k+1/2}y^{a+k+1/2}\int_{0}^{1}\frac{s^{(a-b)/2}(1-s)^{(a+b)/2+k-1}}{(x^2-y^2s)^{(a+b)/2+k+1}}\, ds.
\]
Then, taking $\gamma=(a-b)/2$, $c=b/2+k/2-1/4$, $d=a/2+k/2-1/4$, and $\lambda=1$ in Lemma \ref{lem:estimate}, we deduce that
\[
|K_{a+k}^{b+k}(x,y)|\le C\frac{a+b+2k}{a+k-1/2}\frac{x}{x^2-y^2}\le \frac{C_1}{x-y}, \qquad 0<y<x,
\]
where $C_1$ depends on $a$ and $b$ but not on $k$. Note that our assumptions on $a$, $b$, and $k$ allow us to apply \eqref{eq:int-2F1} and Lemma \ref{lem:estimate} without any problem. To analyze the case $0<x<y$ we proceed in a similar way and the proof of \eqref{eq:T1-size} is completed.

To prove \eqref{eq:T1-smooth}, we check that 
\begin{equation}
\label{eq:T1-smooth-2}
\left|\frac{\partial}{\partial x}K_{a+k}^{b+k}(x,y)\right|\le\frac{C_2}{(x-y)^2}, \qquad 0<y<x,
\end{equation}
only because the other cases can be obtained in a similar way.

From the identity \cite[15.5.1]{NIST}
\[
\frac{d}{dz}{}_2F_1(p,q;r;z)=\frac{pq}{r}{}_2F_1(p+1,q+1;r+1;z)
\]
and \eqref{eq:int-2F1}
we have
\[
\frac{\partial}{\partial x}K_{a+k}^{b+k}(x,y)=-\frac{1}{\Gamma((b-a)/2)\, \Gamma((a-b+2)/2)}(I_1+I_2),
\]
where
\[
I_1=\left(\frac{a+b}{2}+k\right)(2a+2k+3)x^{b+k-1/2}y^{a+k+1/2}\int_{0}^{1}\frac{s^{(a-b)/2}(1-s)^{(a+b)/2+k-1}}
{(x^2-y^2 s)^{(a+b)/2+k+1}}\, ds
\]
and
\[
I_2=\left(\frac{a+b}{2}+k\right)\left(\frac{a+b}{2}+k+1\right)x^{b+k-1/2}y^{a+k+5/2}
\int_{0}^{1}\frac{s^{(a-b+2)/2}(1-s)^{(a+b)/2+k-1}}{(x^2-y^2 s)^{(a+b)/2+k+2}}\, ds.
\]
Using that
\[
I_1\le C k^2x^{b+k-1/2}y^{a+k+1/2}\int_{0}^{1}\frac{s^{(a-b)/2}(1-s)^{(a+b)/2+k-2}}
{(x^2-y^2 s)^{(a+b)/2+k+1}}\, ds
\]
and applying Lemma \ref{lem:estimate} with $\gamma=(a-b)/2$, $c=b/2+k/2-7/4$, $d=a/2+k/2+1/4$, and $\lambda=2$, we have
\[
I_1\le \frac{C_2}{(x-y)^2}.
\]
Finally, taking $\gamma=(a-b+2)/2$, $c=b/2+k/2-7/4$, $d=a/2+k/2+5/4$, and $\lambda=2$ in Lemma \ref{lem:estimate} we conclude that
\[
I_2\le C \frac{((a+b)/2+k)((a+b)/2+k+1)}{(a/2+k/2+5/4)^2}\frac{x^2}{(x^2-y^2)^2}\le \frac{C_2}{(x-y)^2}
\]
and the proof of \eqref{eq:T1-smooth-2} is finished.
\end{proof}

\begin{proof}[Proof of Theorem \ref{thm:main1}]
Let us suppose that $b>a$ (the case $b<a$ can be treated in the same way) and we take $m=\lfloor b-a \rfloor$, where $\lfloor \cdot \rfloor$ denotes the function integer part. Then
\[
S_k^{a,b}f(x)=S_{k}^{a,a+1}\circ S_k^{a+1,a+2}\circ\cdots \circ S_{k}^{a+m-1,a+m}\circ S_k^{a+m,b}f(x). 
\]
From Proposition \ref{prop:C-Z}, $S_{k}^{a+j,a+j+1}$, with $j=0,\dots, m-1$, and $S_k^{a+m,b}$ are uniform Calder\'on-Zygmund operators and the proof of the result follows inmediately from the standard theory.
\end{proof}

\section{Proof of Theorem \ref{thm:JL}}
Taking the mixed-norm spaces
\[
L^{p,2}(\mathbb{R}^n,u(|\cdot|))=\{f: \|f\|_{L^{p,2}(\mathbb{R}^n,u(|\cdot|))}<\infty\},
\]
where
\[
\|f\|_{L^{p,2}(\mathbb{R}^n,u(|\cdot|))}=\left(\int_{0}^{\infty}\left(\int_{\mathcal{S}^{n-1}}|f(r\theta)|^2\, d\sigma(\theta)
\right)^{p/2}r^{n-1}u(r)\, dr\right)^{1/p},
\]
we have the following transference result.
\begin{thm}
\label{thm:mixed-norm}
Let be $n\ge 2$, $n_d=(n+d-1)(1-p/2)$ and $u(r)r^{n_d}\in A_p(0,\infty)$. Then the inequality 
\[
\|T_m f\|_{L^{p,2}(\mathbb{R}^n,u(|\cdot|)|\cdot|^{d(1-p/2)})}\le \|f\|_{L^{p,2}(\mathbb{R}^n, u(|\cdot|)|\cdot|^{d(1-p/2)})} 
\]
implies 
\[
\|T_m f\|_{L^{p,2}(\mathbb{R}^{n+d}, u(|\cdot|))}\le \|f\|_{L^{p,2}(\mathbb{R}^{n+d}, u(|\cdot|))}, \qquad d=1,2,\dots . 
\]
\end{thm}

The proof of Theorem \ref{thm:JL} follows from the case $p=2$ in the previous theorem because 
\[
\|f\|_{L^{2,2}(\mathbb{R}^n, u(|\cdot|))}=\|f\|_{L^{2}(\mathbb{R}^n, u(|\cdot|))}.
\]
 To prove Theorem \ref{thm:mixed-norm} we will apply the next vector-valued inequality.

\begin{prop}
\label{prop:vector-trans}
  Let $a\not =b$, $a,b\ge -1/2$, $1<p<\infty$, and $u\in A_p(0,\infty)$. Then, 
\[
\left\|\left(\sum_{k=0}^\infty|S_k^{a,b} f_k|^2\right)^{1/2}\right\|_{L^p(u)}\le C\left\|\left(\sum_{k=0}^\infty|f_k|^2\right)^{1/2}\right\|_{L^p(u)},
\] 
where the constant $C$ depend on $a$ and $b$.
\end{prop} 

\begin{proof}
For $p=2$ the result follows from Theorem \ref{thm:main1} and for $p\not= 2$ is obtained by extrapolation as in \cite[Theorem 1.1]{Javi-Osane}.
\end{proof}

\begin{proof}[Proof of Theorem \ref{thm:mixed-norm}]
Remember that for suitable functions on $\mathbb{R}^n$
\[
f(x)=\sum_{\begin{smallmatrix}
             k\ge 0\\
             1\le j \le d_k  
           \end{smallmatrix}} f_{k,j}(|x|)\mathcal{Y}_j^{k}\left(\frac{x}{|x|}\right),
\]
where $\{\mathcal{Y}_j^k\}_{k\ge 0, 1\le j\le  d_k}$ is an orthonormal basis of spherical harmonics in $L^2(\mathbb{S}^{n-1})$ and
\[
f_{k,j}(r)=\int_{\mathbb{S}^{n-1}}f(r\theta)\mathcal{Y}_j^k(\theta)\, d\sigma(\theta).
\]
Each $\mathcal{Y}_j^k$ is the restriction to $\mathbb{S}^{n-1}$ of an element of $\mathcal{A}_k$, the class of homogeneous harmonic polynomials of degree $k$. The dimension of $\mathcal{A}_k$ is the integer $d_k$. 

In this way, see \cite[Ch. 4]{SW},
\[
\hat{f}(\xi)=\frac{1}{|\xi|^{(n-1)/2}}\sum_{\begin{smallmatrix}
             k\ge 0\\
             1\le j \le d_k  
           \end{smallmatrix}} i^{-k} \mathcal{H}_{k+(n-2)/2}(f_{k,j}(\cdot)s^{(n-1)/2})(|\xi|)\mathcal{Y}_j^{k}\left(\frac{\xi}{|\xi|}\right)
\]
and
\[
T_mf(x)=\frac{1}{|x|^{(n-1)/2}}\sum_{\begin{smallmatrix}
             k\ge 0\\
             1\le j \le d_k  
           \end{smallmatrix}} \mathcal{T}_m^{k+(n-2)/2}(f_{k,j}(\cdot)s^{(n-1)/2})(|x|)\mathcal{Y}_j^{k}\left(\frac{x}{|x|}\right),
\]
with
\[
\mathcal{H}_\ell(\mathcal{T}_m^{\ell}f)(s)=m(s)\mathcal{H}_\ell f(s). 
\]
Moreover, the inequality
\[
\|T_mf\|_{L^{p,2}(\mathbb{R}^{n+d},u(|\cdot|))}\le C\|f\|_{L^{p,2}(\mathbb{R}^{n+d},u(|\cdot|))}
\]
is equivalent to
\[
\left\|\left(\sum_{\begin{smallmatrix}
             k\ge 0\\
             1\le j \le d_k  
           \end{smallmatrix}}
           |\mathcal{T}_{m}^{k+(n+d-2)/2}f_{k,j}|^2\right)^{1/2}\right\|_{L^p(u(r)r^{n_d})}
\le C\left\|\left(\sum_{\begin{smallmatrix}
             k\ge 0\\
             1\le j \le d_k  
           \end{smallmatrix}}|f_{k,j}|^2\right)^{1/2}\right\|_{L^p(u(r)r^{n_d})}.
\]

Now, by using the identity
\[
\mathcal{T}_m^{k+(n+d-2)/2}f=S_k^{(n+d-2)/2,(n-2)/2}\mathcal{T}_m^{k+(n-2)/2}S_k^{(n-2)/2,(n+d-2)/2}f
\]
and Proposition \ref{prop:vector-trans}, the proof follows immediately because the inequality
\[
\|T_m f\|_{L^{p,2}(\mathbb{R}^n,u(|\cdot|)|\cdot|^{d(1-p/2)})}\le \|f\|_{L^{p,2}(\mathbb{R}^n, u(|\cdot|)|\cdot|^{d(1-p/2)})} 
\]
can be written as
\[
\left\|\left(\sum_{\begin{smallmatrix}
             k\ge 0\\
             1\le j \le d_k  
           \end{smallmatrix}}
           |\mathcal{T}_{m}^{k+(n-2)/2}f_{k,j}|^2\right)^{1/2}\right\|_{L^p(u(r)r^{n_d})}
\le C\left\|\left(\sum_{\begin{smallmatrix}
             k\ge 0\\
             1\le j \le d_k  
           \end{smallmatrix}}|f_{k,j}|^2\right)^{1/2}\right\|_{L^p(u(r)r^{n_d})}.
\]
\end{proof}

\section{Proof of Lemma \ref{lem:estimate}}
This lemma is an extension of \cite[Lemma 5.3]{Ciau-Ron-Lag} including the powers $s^\gamma$. For $\gamma\ge 0$ the proof can be obtained from \cite[Lemma 5.3]{Ciau-Ron-Lag}, but we include it for sake of completeness. In the case $-1<\gamma<0$ some details have to be considered.

From the inequality
\[
\left(\frac{1-s}{A-Bs}\right)^{c+1/2}\le \frac{1}{A^{c+1/2}},
\]
we have to prove that
\begin{equation}
\label{eq:lem-aux}
\int_{0}^{1}\frac{s^\gamma (1-s)^{d-1}}{(A-Bs)^{d+\lambda}}\, ds\le\frac{C_{\gamma,\lambda}}{d^\lambda}\frac{1}{B^d(A-B)^\lambda}.
\end{equation}
With the change of variable $1-s=(A-B)z/B$, we obtain that
\begin{align}
\int_{0}^{1}\frac{ (1-s)^{d-1}}{(A-Bs)^{d+\lambda}}\, ds&=\frac{1}{B^d(A-B)^\lambda}\int_{0}^{B/(A-B)}\frac{z^{d-1}}{(1+z)^{d+\lambda}}\, dz\notag\\ &\le \frac{\Gamma(d)\, \Gamma(\lambda)}{\Gamma(d+\lambda)}\frac{1}{B^d(A-B)^\lambda}\le \frac{C_\lambda}{d^\lambda}\frac{1}{B^d(A-B)^\lambda}\label{eq:lem-aux-2},
\end{align}
where we have applied that $\Gamma(d)/\Gamma(d+\lambda)\sim d^{-\lambda}$.

For $\gamma\ge 0$, using that
\[
\int_{0}^{1}\frac{s^\gamma (1-s)^{d-1}}{(A-Bs)^{d+\lambda}}\, ds\le \int_{0}^{1}\frac{ (1-s)^{d-1}}{(A-Bs)^{d+\lambda}}\, ds
\]
and \eqref{eq:lem-aux-2}, the proof of \eqref{eq:lem-aux} is clear.

For $-1<\gamma<0$, we consider the decomposition
\[
\int_{0}^{1}\frac{s^\gamma (1-s)^{d-1}}{(A-Bs)^{d+\lambda}}\, ds\le \int_{0}^{1/2}\frac{s^\gamma (1-s)^{d-1}}{(A-Bs)^{d+\lambda}}\, ds+\int_{1/2}^{1}\frac{s^\gamma (1-s)^{d-1}}{(A-Bs)^{d+\lambda}}\, ds.
\]
From the inequality
\[
\int_{1/2}^{1}\frac{s^\gamma (1-s)^{d-1}}{(A-Bs)^{d+\lambda}}\, ds\le \int_{0}^{1}\frac{ (1-s)^{d-1}}{(A-Bs)^{d+\lambda}}\, ds,
\]
using \eqref{eq:lem-aux-2}, we obtain the required estimate for the second integral. 
To estimate the first integral we consider some value $1<p<\infty$ such that $\gamma p>-1$ and apply H\"older inequality. Indeed, by \eqref{eq:lem-aux-2},
\begin{align*}
\int_{0}^{1/2}\frac{s^\gamma (1-s)^{d-1}}{(A-Bs)^{d+\lambda}}\, ds&\le \left(\int_{0}^{1/2}\frac{s^{p\gamma }}{1-s}\,ds\right)^{1/p}\left(\int_{0}^{1/2}\frac{(1-s)^{qd-1}}{(A-Bs)^{q(d+\lambda)}}\right)^{1/q}\\&\le C_\gamma\left(\frac{C_\lambda}{d^{q\lambda}B^{qd}(A-B)^{q\lambda}}\right)^{1/q}\le \frac{C_{\gamma,\lambda}}{d^\lambda}\frac{1}{B^d(A-B)^\lambda}
\end{align*}
and the proof of \eqref{eq:lem-aux} is completed.


\end{document}